\documentclass[11pt,oneside]{amsart}
\usepackage{geometry}                
\usepackage{graphicx}
\usepackage{amssymb}

\usepackage{amsmath,amsthm,amscd,amssymb}
\usepackage{latexsym}
\usepackage[colorlinks,citecolor=red,pagebackref,hypertexnames=false]{hyperref}
\usepackage{ulem}
\usepackage{soul}
\usepackage{tikz}
\numberwithin{equation}{section}
\theoremstyle{plain}
\newtheorem{theorem}{Theorem}[section]
\newtheorem{lemma}[theorem]{Lemma}
\newtheorem{corollary}[theorem]{Corollary}

\newtheorem{conjecture}[theorem]{Conjecture}

\theoremstyle{definition}
\newtheorem{definition}[theorem]{Definition}

\theoremstyle{remark}

\newtheorem{case[theorem]}{Case}

\def\bc{\begin{corollary}}
\def\ec{\end{corollary}}
\def\be{\begin{equation}}
\def\ee{\end{equation}}
\def\bast{\begin{eqnarray*} }
\def\east{\end{eqnarray*} }
\def\bea{\begin{eqnarray}}
\def\eea{\end{eqnarray}}

\def\R{\mathbb R}

\def\R{\Bbb R}
\def\D{\Bbb D}
\def\A{\mathcal{A}}

\def\N{\Bbb N}

\def\({\left(}
\def\){\right)}
\def\[{\left[}
\def\]{\right]}
\def\<{\left\langle}
\def\>{\right\rangle}

\author{M. Dannenberg, W. Hagerstrom, G. Hart, A. Iosevich, T. Le, I. Li and N. Skerrett}
\address{Department of Mathematics, University of Rochester, Rochester, NY}
\email{iosevich@math.rochester.edu}
\thanks{The fourth listed author is supported in part by the National Science Foundation grant no. HDR TRIPODS - 1934962 and NSF DMS 2154232. Additional funding for this project was provided by the Schwartz Discover Grant at the University of Rochester. This paper was written under the auspices of Tripods2023 undergraduate research program for undergraduate students.}
\thanks{}
\begin{document}

\title{Buffon Needle Problem Over Convex Sets}

\maketitle

\begin{abstract}
We solve a variant of the classical Buffon Needle problem. More specifically, we inspect the probability that a randomly oriented needle of length $l$ originating in a bounded convex set $X\subset\mathbb{R}^2$ lies entirely within $X$. Using techniques from convex geometry, we prove an isoperimetric type inequality, showing that among sets $X$ with equal perimeter, the disk maximizes this probability.
\end{abstract}

\maketitle

\section{Introduction} 

\vskip.125in 

The isoperimetric inequality in the plane says that along all the sets of perimeter $2 \pi$, the one that maximizes the area is the unit disk. Steiner (\cite{Steiner1838}) made the first progress towards proving this result. He showed that if the maximizing shape exists, it must be the unit disk. The first rigorous proof was given by Hurwitz in 1902 (see \cite{osserman1978isoperimetric} and the references contained therein). This deep and interesting problem lends itself to many variations. 

The version of the isoperimetric inequality we study in this paper, described in detail in Theorem \ref{main} below, is the following. Suppose that a needle of sufficiently small positive length is dropped in a convex set of perimeter $2 \pi$ such that one end of the needle hits any point of the set with uniform probability. We wish to maximize the probability, denoted by Buffon probability, that the other end of the needle is also in the set. We show that if the convex set under consideration is not a disk, and the needle is sufficiently small, then the Buffon probability of this set is smaller than the corresponding Buffon probability of the unit disk. As the reader shall see, our estimates are quantitative. 

\section{Definitions}

\begin{definition}
    Let $x\in\R^2$, and $l>0$. A \textbf{random needle} of length $l$ at $x$ is a directed line segment originating at $x$ whose orientation with respect to the horizontal axis is chosen uniformly from $[0,2\pi)$. A random needle of length 0 at $x$ is the point $x$.
\end{definition}

\begin{definition}
    Let $X\subset\R^2$ be a bounded convex set. The \textbf{pointwise probability} $p_X(x,l)$, $p_X:X\times\R_{\geq0}\to[0,1]$, is the probability that a random needle of length $l$ at $x$ lies within $X$. The \textbf{Buffon probability} $P_X(l)$, $P_X:\R_{\geq0}\to[0,1]$, is the probability that a random needle of length $l$ at $x$ chosen uniformly from $X$ lies within $X$.
\end{definition}

Denote open ball of radius $r$ at $p$ by $B_r(p)$. Recall that for two sets $A,B$, the Minkowski sum is defined by\begin{align*}
    A+B=\{a+b:a\in A,b\in B\}
\end{align*}and the Minkowski difference is defined by\begin{align*}
    A-B=(A^\complement+(-B))^\complement
\end{align*}where $-B=\{-b:b\in B\}$.

\begin{definition}
    Let $X\subset\R^2$ be a bounded convex set. The \textbf{exterior parallel} $X^l$ of $X$ by $l>0$ is the Minkowski sum $X+\overline{B_l(0)}$. The \textbf{interior parallel} $X_l$ of $X$ by $l>0$ is the Minkowski difference $X-\overline{B_l(0)}$.
\end{definition}

Also, we denote area by $\A$ and length by $\ell$.

\section{Main Results}

Our main result is the following isoperimetric type inequality.

\begin{theorem}{\label{main}}
    Let $\D$ be the unit disk. For any compact, convex set $X\in\R^2$ with perimeter $\ell(\partial X)=2\pi$ where $X$ is not a disk, there exists an $\epsilon>0$ such that $P_X(l)<P_\D(l)$ for $l\in(0,\epsilon)$.
\end{theorem}

Our proof of this theorem requires three lemmas. First, due to the symmetries of the disk, $P_\D(l)$ can be exactly computed.

\begin{lemma}{\label{disk}}
    Let $\D$ be the unit disk and $0<l\leq2$. Then\begin{align*}
        P_\D(l)=\frac{2}{\pi}\(\arccos\(\frac{l}{2}\)-\frac{l}{2}\sqrt{1-\frac{l^2}{4}}\).
    \end{align*}
\end{lemma}

The second lemma applies a bound to the pointwise probability. Illustrated below for a convex set with smooth boundary, the red curve is a boundary of $X$, the orange curve is the circle of radius $l$ centered at $x$, and the blue curve is the tangent of $\partial A$ at the point closest to $x$. It can be seen that the convexity of $X$ bounds $p_X(x,l)$ to the proportion of the orange circle below the tangent line.

\begin{figure}[h!]
    \centering
    \includegraphics[width=0.8\linewidth]{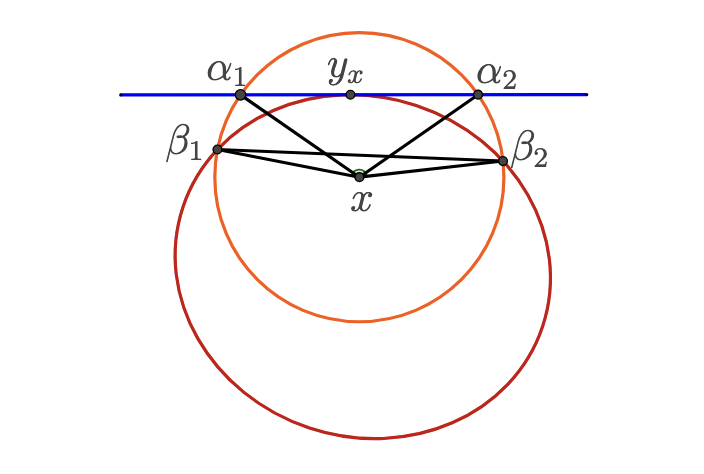}
    \caption{}
    \label{fig:1}
\end{figure}

An adaptation of this argument that handles convex sets without tangents yields the following bound.

\begin{lemma}{\label{pointwise}}
    Let $X\subset\R^2$ be a bounded convex set with perimeter $2\pi$. For $l>0$, we have\begin{align*}
        \int_{X\setminus X_l}p_X(x,l)dx\leq2\pi l-2l.
    \end{align*}
\end{lemma}

In the final step, we split $X$ into $X_l$ and $X\setminus X_l$. Using the third lemma (an extension of Steiner's formulae \cite{SteinForm}), we obtain an upper bound on $\A(X_l)/\A(X)$ in terms of $l$, $\A(X)$, and $\ell(\partial(X_l)^l)$.

\begin{lemma}{\label{steiner}}
    Let $X$ be a bounded convex subset of $\R^2$ and $r>0$. Then\begin{align*}
        \A(X)&\geq\A((X_r)^r)=\pi r^2+\ell(\partial X_r)r+\A(X_r), \\
        \ell(\partial X)&\geq\ell(\partial(X_r)^r)=2\pi r+\ell(\partial X_r),
    \end{align*}and if $X$ is not a line segment,\begin{align*}
        \lim_{r\to0}\ell(\partial(X_r)^r)=\lim_{r\to0}\ell(\partial X_r)=\ell(\partial X).
    \end{align*}
\end{lemma}

Using the second lemma, we can get an upper bound on\begin{align*}
    \frac{\int_{X\setminus X_l}p_X(x,l)}{\A(X)}
\end{align*}in terms of $l$ and $\A(X)$. Summing these bounds yields a bound on $P_X(l)$. Using the isoperimetric inequality\cite{osserman1978isoperimetric}, it can be shown that this upper bound is less than $P_\D(l)$ for sufficiently small $l$.

\begin{theorem}[Isoperimetric Inequality]{\label{iso}}
    Let $C$ be a simple curve from $\R^2$ that encloses a region $X$. Then\begin{align*}
        \ell(C)^2\geq4\pi\A(X)
    \end{align*}where equality holds if and only if $C$ is a circle.
\end{theorem}

\section{Proof of Main Result}

\begin{proof}[Theorem \ref{main}]
    If $X$ is a line segment, $P_X(l)=0$ for all $l>0$, so $P_X(l)<P_\D(l)$ for all $0<l<2$. As such, assume $X$ is not a line segment. By Lemma \ref{pointwise},\begin{align*}
        P_X(l)=\frac{1}{\A(X)}\int_Xp_X(x,l)dx&=\frac{1}{\A(X)}\int_{X_l}p_X(x,l)dx+\frac{1}{\A(X)}\int_{X\setminus X_l}p_X(x,l)dx \\
        &\leq\frac{\A(X_l)}{\A(X)}+\frac{2\pi l-2l}{\A(X)}
    \end{align*}and by Lemma \ref{steiner},\begin{align*}
        \A(X_l)\leq\A(X)+\pi l^2-\ell(\partial(X_l)^l)l.
    \end{align*}
    Therefore\begin{align*}
        P_X(l)\leq\frac{\A(X)+\pi l^2-2l+l(2\pi-\ell(\partial(X_l)^l))}{\A(X)},
    \end{align*}
    so by Lemma \ref{disk},\begin{align*}
        P_\D(l)-P_X(l)&\geq\frac{2}{\pi}\(\arccos\(\frac{l}{2}\)-\frac{l}{2}\sqrt{1-\frac{l^2}{4}}\)-\frac{\A(X)+\pi l^2-2l+l(2\pi-\ell(\partial(X_l)^l))}{\A(X)} \\
        &=h(l)+l\frac{\ell(\partial(X_l)^l)-2\pi}{\A(X)}
    \end{align*}where $h:(-2,2)\to\R$ is given by\begin{align*}
        h(l)=\frac{2}{\pi}\(\arccos\(\frac{l}{2}\)-\frac{l}{2}\sqrt{1-\frac{l^2}{4}}\)-\frac{\A(X)+\pi l^2-2l}{\A(X)}.
    \end{align*}
    Note that\begin{align*}
        h'(l)=\frac{2-2l\pi}{\A(X)}-\frac{\sqrt{4-l^2}}{\pi},
    \end{align*}so $h'(0)=\frac{2}{\A(X)}-\frac{2}{\pi}>0$ by Theorem \ref{iso}. By Lemma \ref{steiner}, we can choose $\delta>0$ such that\begin{align*}
        \left|\frac{\ell(\partial(X_l)^l)-2\pi}{\A(X)}\right|<\frac{h'(0)}{2}
    \end{align*}when $0<l<\delta$. Then for such an $l$,\begin{align*}
        P_\D(l)-P_X(l)\geq h(l)-l\frac{h'(0)}{2},
    \end{align*}
    but since $P_\D(0)=P_X(0)=1$, this inequality actually holds for $l\in[0,\delta)$. The derivative of the LHS with respect to $l$ is $h'(l)-\frac{h'(0)}{2}$. Evaluating at 0, we get $\frac{h'(0)}{2}>0$. Thus, for some $\epsilon>0$, we have\begin{align*}
        0<h(l)-l\frac{h'(0)}{2}\leq P_\D(l)-P_X(l)
    \end{align*}for $l\in(0,\epsilon)$. This is the desired result.
\end{proof}

\section{Proofs of Lemmas}

\begin{proof}[Lemma \ref{disk}]
    Let $\Omega$ be the set of oriented needles of length $l$ originating in $\D$. Define $\chi:\Omega\to[0,1]$ by $\chi(n)=1$ if both endpoints of $n$ lie in $\D$, and 0 otherwise. Then, if we let $dK$ denote the kinematic density \cite{Santaló_2004} in $\R^2$, our Buffon probability is given by\begin{align*}
        P_\D(l)=\frac{\int_\Omega\chi dK}{\int_\Omega dK}.
    \end{align*}
    Represent $n\in\Omega$ by $(x,\theta)\in\R^2\times[0,2\pi)$, where $n$ originates at $x$ with angle $\theta$ with respect to the horizontal axis. Then we have that $dK=dx\wedge d\theta$, so\begin{align*}
        P_\D(l)=\frac{\int_0^{2\pi}\int_\D\chi(x,\theta)dxd\theta}{\int_0^{2\pi}\int_\D dxd\theta}=\frac{1}{2\pi^2}\int_0^{2\pi}\int_\D\chi(x,\theta)dxd\theta.
    \end{align*}
    Note that\begin{align*}
        \int_\D\chi(x,\theta)dx=\int_\D\chi(x,0)dx.
    \end{align*}
    Let $\D^+$ and $\D^-$ denote the upper and lower half-disks. Then\begin{align*}
        \int_\D\chi(x,0)dx&=\int_{\D^+}\chi(x,0)dx+\int_{\D^-}\chi(x,0)dx \\
        &=2\int_{\D^+}\chi(x,0)dx.
    \end{align*}
    For $x=(x_1,x_2)\in\D^+$, $\chi((x_1,x_2),0)$ is 1 if $-\sqrt{1-x_2^2}\leq x_1\leq\sqrt{1-x_2^2-l}$. Thus\begin{align*}
        \int_{\D^+}\chi(x,0)dx&=\int_0^{\sqrt{1-\(\frac{l}{2}\)^2}}\int_{-\sqrt{1-x_2^2}}^{\sqrt{1-x_2^2-l}}dx_1dx_2 \\
        &=\int_0^{\sqrt{1-\(\frac{l}{2}\)^2}}2\sqrt{1-x_2^2}-ldx_2 \\
        &=\arccos\(\frac{l}{2}\)-\frac{l}{2}\sqrt{1-\frac{l^2}{4}}.
    \end{align*}
    As such,\begin{align*}
        P_\D(l)=\frac{2}{\pi}\(\arccos\(\frac{l}{2}\)-\frac{l}{2}\sqrt{1-\frac{l^2}{4}}\).
    \end{align*}
\end{proof}

\begin{proof}[Lemma \ref{pointwise}]
    Fix $x\in X\setminus(X_l\cup\partial X)$. Let $y_x\in\partial X$ be such that $\inf_{z\in\partial X}|x-z|=|x-y_x|$. Now, let $p$ be the endpoint of a needle of length $l$ originating at $x$ and rotated $\arccos\(\frac{|x-y_x|}{l}\)$ radians clockwise and let $q$ be the endpoint obtained by rotating counterclockwise. Then, $p$, $q$, and $y_x$ are colinear. We can assume that $x$ is the origin and $y_x$ is on the vertical axis. By this assumption, the coordinates of our points are\begin{align*}
        x&=(0,0),&q&=\(-\sqrt{l^2-|y_x-x|^2},|y_x-x|\),\\y_x&=(0,|y_x-x|),&p&=\(\sqrt{l^2-|y_x-x|^2},|y_x-x|\).
    \end{align*}
    Let $a$ denote the arc between $p$ and $q$ of length $2\arccos\(\frac{|y_x-x|}{l}\)$ that goes counterclockwise from $p$ to $q$. Suppose by way of contradiction that some point $z\neq p,q$ is countained in $X\cap a$. If the first coordinate of $z$ is zero, choose points $x'$ and $x''$ in $X$ to the left and right of $x$. Then the triangle $\triangle x'x''z$ are contained in $X$, but this interior contains $y_x$, which is a contradiction, since $y_x$ is on the boundary of $X$! If the first coordinate of $z$ is nonzero, assume without loss of generality that the first coordinate is positive. Consider the line segment $\overline{zy_x}$, which is contained in $X$ by convexity. Convexity and the assumption that $y_x\in\partial X$ restricts the boundary of $X$ to the second quadrant. If $(b_1,b_2)\in X$, $b_1<0$ and $b_2\geq0$, we have $b_2\leq Z(b_1)$ where $Z$ is the graph obtained by extending the segment $\overline{zy_x}$. Note that a circle of radius of strictly less than $l$ centered at $x$ intersects $Z$ in the second quadrant, hence there is a point in $\partial A$ that is strictly closer to $x$ than $y_x$. This is a contradiction, since $|y_x-x|$ is minimal! Thus, the arc $a$ only intersects $X$ at the points $p$ and $q$. Hence\begin{align*}
        p_X(x,l)\leq\frac{1}{2\pi}\(\pi+2\arcsin\(\frac{|x-y_x|}{l}\)\).
    \end{align*}
    Note that, for a fixed $t$, $p_X(q,l)$ is bounded above for $q\in\partial X_t$ by\begin{align*}
        g_X(t,l)=\frac{1}{2\pi}\(\pi+2\arcsin\(\frac{t}{l}\)\).
    \end{align*}
    Since $X_t\subset X$, we have $\ell(\partial X_l)\leq\ell(\partial X)$, so\begin{align*}
        \int_{X\setminus X_l}p_X(x,l)dxdt&=\int_{X\setminus(X_l\cup\partial A)} \\
        &\leq\int_0^l\ell(\partial X_t)g_X(t,l)dt \\
        &\leq\int_0^l\(\pi+2\arcsin\(\frac{t}{l}\)\) \\
        &=2\pi l-2l,
    \end{align*}as desired. Note that the first inequality holds because $\partial A$ has measure 0.
\end{proof}

For the proof of Lemma \ref{steiner}, we first state Steiner's formlae \cite{SteinForm}.

\begin{theorem}[Steiner's Formulae]
    Let $X$ be a compact convex subset of $\R^2$. For $r>0$,\begin{align*}
        \A(X^r)&=\pi r^2+\ell(\partial X)r+\A(X), \\
        \ell(\partial X^r)&=2\pi r+\ell(\partial X).
    \end{align*}
\end{theorem}

\begin{proof}[Lemma \ref{steiner}]
    Let $X$ be a compact convex subset of $\R^2$. It follows from the definitions of interior and exterior parallels that $(X_r)^r\subset X$. Since $X_r$ is convex, Steiner's formulae give\begin{align*}
        \A((X_r)^r)&=\pi r^2+\ell(\partial X_r)r+\A(X_r), \\
        \ell(\partial(X_r)^2)&=2\pi r+\ell(\partial X_r).
    \end{align*}
    Since $(X_r)^r$ is convex and contained in $X$, $\A(X)\geq\A((X_r)^r)$ and $\ell(\partial X)\geq\ell(\partial(X_r)^r)$. As such,\begin{align*}
        \A(X)&\geq\pi r^2+\ell(\partial X_r)r+\A(X_r), \\
        \ell(\partial X)&\geq 2\pi r+\ell(\partial X_r).
    \end{align*}
    This is the first of the desired results.

    Now, assume that $X$ is not a line or line segment, and note that\begin{align*}
        \bigcup_{n\in\N}X_{1/n}=X\setminus\partial X.
    \end{align*}
    Let $f:[0,1]\to\R^2$ be a parameterization of $\partial X$. Fix some $\epsilon>0$. Since $\partial X$ is bounded and convex, it is rectifiable, so we can take some partition $\{t_0,\ldots,t_k\}$ of $[0,1]$ such that\begin{align*}
        \left|\ell(\partial X)-\sum_{i=1}^k|f(t_i)-f(t_{i-1})\right|<\epsilon.
    \end{align*}
    Since $X$ is not a line or line segment, we can assume that 0 is in the interior of $X$. Then for $0<c<1$, the set $cX=\{cx:x\in X\}$ is convex and contained in $X\setminus\partial X$. Since $\{f(t_1),\ldots,f(t_k)\}$ is a finite set, we can choose $c$ large enough such that $|cf(t_i)-f(t_i)|<\frac{\epsilon}{k}$. Since $X_{1/n}\subset X_{1/(n+1)}$ and $\{cf(t_1),\ldots,cf(t_k)\}\subset X\setminus\partial X$, it must be that $\{cf(t_1),\ldots,cf(t_k)\}\subset X_{1/N}$ for some $N$. Since $A_{1/N}$ is convex, it contains the convex hull of the points $cf(t_i)$. Denote this set by $C$. Now we have\begin{align*}
        \ell(\partial A_{1/N})&\geq\ell(\partial C) \\
        &=\sum_{i=1}^k|cf(t_i)-cf(t_{i-1})| \\
        &\geq-2\epsilon+\sum_{i=1}^k|f(t_i)-f(t_{i-1})| \\
        &\geq\ell(\partial X)-3\epsilon.
    \end{align*}
    Since $\ell(\partial X_{1/n})$ is an increasing sequence bounded above by $\ell(\partial X)$, the above shows that we must have $\lim_{n\to\infty}\ell(\partial X_{1/n})=\ell(\partial X)$. The other limit follows from the inequality\begin{align*}
        \ell(\partial X_r)\leq\ell(\partial(X_r)^r)\leq\ell(\partial X).
    \end{align*}
\end{proof}

\section{Simulations and Future Work}

\begin{figure}[h!]
    \centering
    \includegraphics[width=0.5\linewidth]{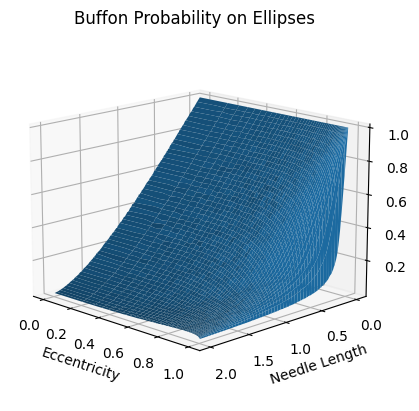}
    \caption{Buffon Probability by Eccentricity}
    \label{fig:2}
\end{figure}

Figure \ref{fig:2} graphs the simulated Buffon probability of ellipses on the vertical axis against their eccentricity and needle length on the horizontal axes. This graph suggests that the Buffon probability of the circle begins to dominate for needle lengths around $\frac{3}{2}$. These numerical results suggest the existence of a global result stronger than Theorem \ref{main}. With this in mind, we make the following conjecture.

\begin{conjecture}
    Let $\D$ be the unit disk. There exists some $\delta>0$ such that for any non-disk, compact convex set $X\in\R^2$ where $\ell(\partial X)=2\pi$ and $0<l<\delta$, we have $P_\D>P_X(l)$.
\end{conjecture}

The existence of this sharper result is also expected from a theoretical standpoint, since the convexity bound used in Lemma \ref{pointwise} is not sharp. For this reason, we expect that future work towards this result will rely on a closer inspection of the boundary of compact, convex sets $X\in\R^2$. This bound in Lemma \ref{pointwise} assumes the worst local case, that the boundary is flat. Assuming that $\partial X$ is sufficiently smooth, the global properties of its curvature should yield more precise bounds on $P_X(l)$.

Another possible direction of future work is the extension of this result to multiple directions. We state here a generalization of Steiner's formulae to multiple dimensions \cite{Dacorogna}.

\begin{theorem}(Minkowski-Steiner Formula)
    Let $X\subset\R^n$ be a convex set, $\mu$ denote the $n$ dimensional Lebesgue measure (volume), and $\lambda$ denote the $(n-1)$ dimensional measure. Then\begin{align*}
        \mu\(X+\overline{B_\delta}\)=\mu(X)+\lambda(\partial X)\delta+\sum_{i=2}^{n-1}\lambda_i(X)\delta^i+\omega_n\delta^n,
    \end{align*}where $\lambda_i$ are quermassintegrals of $X$ and $\omega_n$ denotes the measure of the unit ball in $\R^n$.
\end{theorem}


\begin{thebibliography}{99}

\bibitem{SteinForm}
Joseph Ansel Hoisington.
\newblock Steiner's formula and a variational proof of the isoperimetric inequality.
\newblock \emph{arXiv preprint}, 1909.06347, 2019.
  
\bibitem{Santaló_2004}
Luis A. Santaló.
\newblock \emph{Integral Geometry and Geometric Probability}.
\newblock 2nd edition, Cambridge Mathematical Library, Cambridge University Press, 2004.

\bibitem{Dacorogna}
Bernard Dacorogna.
\newblock Introduction to the Calculus of Variations.
\newblock \emph{2009}.

\bibitem{Steiner1838}
Jakob Steiner.
\newblock Einfacher Beweis der isoperimetrischen Hauptsätze.
\newblock \emph{Journal für die reine und angewandte Mathematik}, 18:281--296, 1838.

\bibitem{osserman1978isoperimetric}
Robert Osserman.
\newblock The Isoperimetric Inequality.
\newblock \emph{Bulletin of the American Mathematical Society}, 84(6):1182--1238, 1978.
\newblock \url{http://www.ams.org/journals/bull/1978-84-06/S0002-9904-1978-14553-4/S0002-9904-1978-14553-4.pdf}.

\end{thebibliography}
\end{document}